\documentclass[12pt, reqno]{amsart}
\usepackage{amsmath}
\usepackage{amssymb}
\usepackage{amsthm}
\usepackage{amscd}
\usepackage{amsxtra}
\usepackage{verbatim}
\usepackage{color}
\usepackage{enumerate}
\usepackage{mathrsfs}

\numberwithin{equation}{section}

 \RequirePackage{geometry}
\newtheorem{theorem}{Theorem}[section]
\newtheorem{lemma}[theorem]{Lemma}
\newtheorem{corollary}[theorem]{Corollary}
\newtheorem{proposition}[theorem]{Proposition}

\newtheorem{problem}[theorem]{Problem}

\theoremstyle{definition}

\newtheorem{remark}[theorem]{Remark}
\newtheorem{definition}[theorem]{Definition}
\newtheorem{example}[theorem]{Example}

\theoremstyle{remark}

\newcommand{\spn}{\mathop\mathrm{span}}

\DeclareMathOperator{\rank}{\mathrm{rank}}

\DeclareMathOperator{\tr}{\mathrm{tr}}
\DeclareMathOperator{\Imm}{\mathrm{Im}}

\newcommand{\PARENTH}[1]{ \left( #1 \right) } 

\newcommand{\INDEX}[1]{\big[{[}#1{]}\big]}

\begin{document}
	
	\title{A notion of optimal packings of subspaces with mixed-rank and solutions}

	\author[Casazza, Haas, Stueck, Tran
	]{Peter G. Casazza, John I. Haas IV, Joshua Stueck, and Tin T. Tran\\
		(In memory of John I. Haas)}

	\address{219 Mathematical Sciences Building, Department of Mathematics, University of Missouri, Columbia, MO 65211}
	\email{Casazzap@missouri.edu}
	\email{haasji@missouri.edu}
	\email{jsstueck@mail.mizzou.edu}
	\email{tttrz9@mail.missouri.edu}

	\begin{abstract} We resolve a longstanding open problem by reformulating the Grassmannian fusion frames to the case of mixed dimensions and show that this satisfies the proper properties for the problem. In order to compare elements of mixed dimension, we use a classical embedding to send all fusion frame elements to points on a higher dimensional Euclidean sphere, where they are given ``equal footing''.  Over the embedded images -- a compact subset in the higher dimensional embedded sphere -- we define optimality in terms of the corresponding restricted coding problem. We then construct infinite families of solutions to the problem by using maximal sets of mutually unbiased bases and block designs. Finally, we show that using Hadamard 3-designs in this construction leads to infinite examples of maximal orthoplectic fusion frames of constant-rank. Moreover, any such fusion frames constructed by this method must come from Hadamard 3-designs.
\end{abstract}

	\maketitle
\section{Introduction}
Let $\mathcal{X}$ be a compact metric space endowed with a distance function $d_\mathcal{X}$. The packing problem is the problem of finding a finite subset of $\mathcal{X}$ so that the minimum pairwise distance between points of this set is maximized. When $\mathcal{X}$ is the Grassmannian manifold and $d_\mathcal{X}$ is the chordal distance, the problem has received considerable attention over the last century. This has been motivated in part by emerging applications such as in quantum state tomography~\cite{RKSC, Z}, compressed sensing~\cite{BFMW} and coding theory~\cite{BBC, SH, XZG}.  Much previous work also focuses on the special case where the Grassmannian manifold contains subspaces of dimension one. This direction became an active area of research in the context of frame theory, see \cite{BK1, BoE, BH2, BH3, Et1, FMT, FJMP, FJM, HS, HP, JMF, K, O, STDH, SH}, for instance.  

In this work, we will consider the problem of optimal packings of subspaces with multiple dimensions.  We expect that this generalized notion will similarly be useful in the signal processing realm, but our emphasis is on upholding the underlying geometric tenet as a packing problem.

The paper is organized as follows. After preliminaries in Section~\ref{sec2}, we will take a look at the constant-rank packing problem in Section \ref{sec3}. In particular, after dissecting the well-known proof for the derivation of the (chordal) simplex and orthoplex bounds, we can see that solutions for this problem can be found by solving a corresponding restricted coding problem in a higher dimensional space. This motivates us to give a definition for optimal packings of subspaces with multiple dimensions in Section \ref{sec4}. Next, we present some properties of solutions of the mixed-rank packing problem in Section \ref{sec5}. In particular, as in the constant-rank case, we show that the solutions to the problem always form fusion frames. Finally, we will construct some infinite families of solutions in Section \ref{sec6}. We also focus on constructing maximal orthoplectic fusion frames of constant-rank in this section.

\section{Preliminaries}\label{sec2}
Fix $\mathbb{F}$ as $\mathbb{R}$ or $\mathbb{C}$; furthermore, prescribe $d, l, m, n \in \mathbb{N}$, where the last three numbers satisfy $l \le m \le n$. We also denote $\INDEX{m}$ for the set $\INDEX{m}:=\{1, 2, \ldots, m\}$.
A {\bf code} is a sequence of $n$ points on the real unit sphere, $\mathcal{S}^{d-1} \subset \mathbb{R}^d$.    
A {\bf packing} is a sequence of $m \times m$ orthogonal projections onto subspaces of $\mathbb{F}^m$, which we often identify with their associated subspaces; that is, we frequently identify a given packing $\mathcal{P}:=\{P_i\}_{i=1}^n$ with its subspace images, $\mathcal{P}:= \{W_i:=\Imm(P_i)\}_{i=1}^n$. 
The {\bf mixture} of a packing is the number of different ranks occurring among the packing's elements.
If a packing has a mixture of one, meaning the ranks of a packing's elements equate, then the packing is {\bf constant-rank}; otherwise, it is {\bf mixed-rank}. 
We define the {\bf Grassmannian manifold}, $\mathcal{G}(l, \mathbb{F}^m)$, as the set of all orthogonal projections over $\mathbb{F}^m$  with rank $l$.  Thus, a constant-rank packing, $\mathcal{P}$, is a sequence of points in the $n$-fold Cartesian product of Grassmannian manifolds:  
$$\mathcal{P}=\{P_i\}_{i=1}^n \in \prod_{i=1}^n \mathcal{G}(l, \mathbb{F}^m).$$
Conventionalizing the preceding notation, we reserve the symbol, $d$, for the ambient dimension of a given code,  we reserve, $m$, for the ambient dimension of a given packing, and $n$ denotes the number of points in an arbitrary code or packing.

The existence of a mixed-rank packing, $\mathcal{P}$, implicates the existence of a sequence of pairs, $\{(n_i, l_i)\}_{i=1}^{s}$, where each $n_i$ denotes the number of packing elements of rank $l_i$, meaning $n = \sum_{i=1}^s n_i$, and $1\leq l_i\leq m$, for $i\in \INDEX{s}$.  This leads to a dictionary-type ordering on the elements, which -- as in the constant-rank case -- respects the interpretation of mixed-rank packings as sequences in a Cartesian product of Grassmannian manifolds:
$$
\mathcal{P}=\{P_{j}^{(l_i)}\}_{i=1, j=1}^{s,\ \ n_i} \in \prod_{i=1}^s \prod_{j=1}^{n_i} \mathcal{G}(l_i, \mathbb{F}^m).
$$   
A {\bf fusion frame}, $\mathcal{P} = \{P_i\}_{i=1}^n$,  is a packing such that its {\bf fusion frame operator}, 
$$ F_P := \sum_{i=1}^n P_i,$$ 
is positive-definite. If the $F_P$ is a multiple of the $m \times m$ identity operator, henceforth denoted $I_m$, then $\mathcal{P}$ is {\bf tight.} A constant-rank fusion frame for which all elements are  rank-one is a {\bf frame}. For more information on (fusion) frames, we recommend \cite{CG, C, W}.

\section{Constant-rank packings revisited}\label{sec3}
In order to generalize the notion of optimal packings of subspaces with different dimensions, we will take a look at the constant-rank packing problem as discussed, for example, in \cite{BH1, CHST, CHS}.
\begin{definition}
	Given two $l$-dimensional subspaces of $\mathbb{F}^m$ with corresponding orthogonal projections $P$ and $Q$, the chordal distance between them is 
	\[d_c(P, Q):=\dfrac{1}{\sqrt{2}}\Vert P-Q\Vert_{H.S}=(l-\tr(PQ))^{1/2}.\]
\end{definition}

\begin{definition}
	Given a constant-rank packing, $\mathcal{P}=\{P_i\}_{i=1}^n \subset \mathcal{G}(l, \mathbb{F}^m)$, its {\bf (chordal) coherence} is
	$$
	\mu(\mathcal{P}): = \max_{\substack{1 \le i, j \le n \\ i \neq j}} \tr (P_i P_j).
	$$
	The constant-rank packing, $\mathcal{P}$, is {\bf optimally spread} if:
	$$
	\mu(\mathcal{P}) = 
	\min_{   
		\substack{\mathcal{P}'=\{P_i'\}_{i=1}^n\\ \mathcal{P}' \subset \mathcal{G}(l, \mathbb{F}^m)}  
	}
	\mu(\mathcal{P}').
	$$
\end{definition}
For reasons clarified below, we denote the dimension, $d_{\mathbb{F}^m}$, as a function of $m$ and $\mathbb{F}$ as follows\begin{equation}\label{eq:embdim}
d_{\mathbb{F}^m}:=\left\{ \begin{array}{cc} \frac{m(m+1)}{2}-1, & \mathbb{F} = \mathbb{R}\\ m^2-1, & \mathbb{F} = \mathbb{C}
\end{array} \right..
\end{equation}
We recall two well-known bounds for the coherence $\mu(\mathcal{P})$, the simplex and orthoplex bounds, as derived by  Conway, Hardin and Sloane~\cite{CHS} from the  Rankin bound~\cite{R}:
\begin{theorem}[\cite{CHS}]\label{thm1}
	\begin{enumerate}
		\item {\bf Simplex bound}:
		If $\mathcal{P}=\{P_i\}_{i=1}^n \subset \mathcal{G}(l, \mathbb{F}^m)$ is a constant-rank packing for $\mathbb{F}^m$,  
		then 
		$$
		\mu(\mathcal{P}) \geq \frac{nl^2-ml}{m(n-1)}.
		$$
		and equality is achieved if and only if the fusion frame is  equiangular and tight.
		\item {\bf Orthoplex bound}:
		If $\mathcal{P}=\{P_i\}_{i=1}^n \subset \mathcal{G}(l, \mathbb{F}^m)$ is a constant-rank packing for $\mathbb{F}^m$ and  $n >  d_{\mathbb{F}^m}+1$, then
		$$
		\mu(\mathcal{P}) \geq \frac{l^2}{m},
		$$
		and if equality is achieved, then $\mathcal{P}$ is optimally spread and $n\leq 2d_{\mathbb{F}^m}.$
	\end{enumerate}
	
\end{theorem}  

Among constant-rank packings, it seems the most commonly known solutions to this problem arise as equiangular tight fusion frames (ETFFs)~\cite{CHRSS, DHST, King1, KPCL}, including the special case where all projections of packings are of rank one as mentioned. ETFFs are characterized by several special properties, including (i) equiangular, meaning constant modulus of the pair-wise inner products between elements, (ii) tightness, and, of course, (iii) coherence minimized to the  simplex bounds in the previous theorem. When the cardinality of a constant-rank packing is sufficiently high, the orthoplex bound characterizes several other infinite families of optimally spread constant-rank solutions, including maximal sets of mutually unbiased bases~\cite{WF} among others~\cite{BH2, BH3, BH1, CHRSS, CHST, SL}.  Although unnecessary for solutions characterized by the orthoplex bound (see \cite{CH} for a discussion of this phenomenon), many solutions arising from this bound are also tight. 

One way to prove Theorem \ref{thm1} is using a classical embedding to send all fusion frame elements to points on a higher dimensional Euclidean sphere. The desired claims then follow from a result of Rankin:
\begin{theorem}[\cite{CHS}, \cite{R}]\label{thm2}
	Let $d$ be a positive integer and $\{v_1, v_2, \ldots, v_n\}$ be $n$ vectors on the unit sphere in $\mathbb{R}^d$. Then 
	\[\min_{{i\not= j}}\Vert v_i-v_j\Vert\leq \sqrt{\dfrac{2n}{n-1}},\] and if equality is achieved, then $n\leq d+1$ and the vectors form a simplex. Additionally, if $n>d+1$, then the minimum Euclidean distance improves to:
	\[\min_{{i\not= j}}\Vert v_i-v_j\Vert\leq \sqrt{2},\] and if equality holds in this case, then $n\leq 2d$. Moreover, if $n=2d$, then equality holds if and only if the vectors forms an orthoplex, the union of an orthonormal basis with the negatives of its basis vectors.
\end{theorem}
In terms of the inner products between $n$ unit vectors in $\mathbb R^d$, the Rankin bound is $\max_{i\not=j}\langle v_i, v_j\rangle\geq -\frac{1}{n-1}$, and if $n>d+1$, then $\max_{i\not=j}\langle v_i, v_j\rangle\geq 0$.

To see Theorem \ref{thm2} implies Theorem \ref{thm1}, recall the well-known, isometric embedding, which maps constant-rank projections to points on a real, higher dimensional sphere, see \cite{BH1, CHST, CHS}.  Given any $m \times m$ rank $l$ orthogonal projection, $P$, the {\bf $l$-traceless map}, $T_l$, is defined  and denoted as: 
$$   
T_l(P)= P - \frac l m I_m.
$$
As proven in \cite{BH1, CHS}, $T_l$ isometrically injects $P$ into the $d_{\mathbb{F}^m}$-dimensional subspace of the $m \times m$ symmetric/hermitian matrices, $$\mathbb{H} = \{ A \in \mathbb{F}^{m \times m}: \tr(A)=0\},$$ endowed with the standard Hilbert-Schmidt inner product.  Dimension counting yields
the value of $d_{\mathbb{F}^m}$ as function of $m$, as described in Equation~\ref{eq:embdim}. With an appropriate choice of isomorphism, which we record as $\mathcal{V}$, we interpret the inner product between its elements as the standard Euclidean inner product between vectors in $\mathbb{R}^{d_{\mathbb{F}^m}}$. 

Denote by $\Omega_l$ the image of $\mathcal{G}(l, \mathbb{F}^m)$ under the $l$-traceless map, $T_l$ as:
$$
\Omega_l := T_l\PARENTH{\mathcal{G}\PARENTH{l, \mathbb{F}^m} } \subset \mathbb{H}.
$$

It is clear that the image $\Omega_l$  lies on the sphere in $\mathbb{R}^{d_{\mathbb{F}^m}}$, with squared radius 
\begin{equation}\label{eq:radii}
r_l^2=\tr\PARENTH{ \PARENTH{ P-\frac{l}{m}I_m}^2 } 
= \frac{l(m -l)}{m}.\end{equation}

Normalizing the image, $\Omega_l$, to lie on the surface of the unit sphere 
$$
\mathcal{K}_l : = \left\{ \frac{ \mathcal{V}(A) }{r_l}: A \in \Omega_l \right\} \subset \mathcal{S}^{d_{\mathbb{F}^m}-1} \subset \mathbb{R}^{d_{\mathbb{F}^m}},
$$
thereby converting any constant-rank packing into a code. Because $\mathcal{V}$ and the $l$-traceless map are all continuous actions, and because the Grassmannian manifold is well-known to be compact, elementary topological theory implies that $\mathcal{K}_l$ is compact.

For any orthogonal projection $P\in \mathcal{G}(l, \mathbb{F}^m)$, set $v_P=\dfrac{\mathcal{V}(T_l(P))}{r_l}$, then $v\in \mathcal{K}_l.$ We will say that $v_P$ is the {\bf embedded vector} corresponding to $P$. A simple computation yields the identity: 
\begin{equation}\label{eq:const}
\tr \PARENTH{PQ} = \frac{l^2}{m} + \frac{l(m-l)}{m}\langle v_P, v_Q\rangle, 
\end{equation}
for any $P, Q\in \mathcal{G}(l, \mathbb{F}^m)$, where $v_P$ and $v_Q$ are the corresponding embedded vectors of $P$ and $Q$, respectively. Theorem \ref{thm1} then follows by using Equation \ref{eq:const} and Rankin's result.

Thus, the optimality of the packing $\mathcal{P}$ has a close connection with the restricted coding problem. More precisely, for any compact set $\mathcal{K}$ in the unit sphere in $\mathbb{R}^d$, we can consider the following restricted coding problem: 
\begin{problem}
	A code $\mathcal{C}=\{v_i\}_{i=1}^n\subset \mathcal{K}$ is said to be a solution to the restricted coding problem respective to $\mathcal{K}$ if it satisfies:
	$$
	\max_{i \neq j} \langle v_i, v_j \rangle = \min_{ \{u_i\}_{i=1}^n\subset \mathcal{K} } \max_{i \neq j} \langle u_i, u_j \rangle.
	$$
\end{problem}

By what we have discussed, the following proposition is obvious.

\begin{proposition}\label{pro1}
	A packing $\mathcal{P}=\{P_i\}_{i=1}^n\subset \mathcal{G}(l, \mathbb{F}^m)$ is optimally spread if and only if the corresponding embedded unit vectors $\{v_i\}_{i=1}^n \subset \mathbb{R}^{d_{\mathbb{F}^m}}$ are a solution to the restricted coding problem respective to $\mathcal{K}_l$.
\end{proposition}
\section{Generalized to the mixed-rank packing problem}\label{sec4}

In the previous section, we have seen that solutions for constant-rank packings can be found by solving the corresponding restricted coding problem. We will use this idea to define a notion of optimal packings of subspaces with various dimensions.

Given a sequence of pairs $\{(n_i, l_i)\}_{i=1}^{s}$. By the previous section, for each $i$, $\mathcal{K}_{l_i}:=\dfrac{1}{r_{l_i}}\mathcal{V}(T_{l_i}(\mathcal G\PARENTH{l_i, \mathbb F^m}))$ is a compact set in the unit sphere of $\mathbb{R}^{d_{\mathbb{F}^m}}$. Note that  $d_{\mathbb{F}^m}$ is independent of  $l_i$. Moreover, they are disjoint sets.
\begin{proposition}
	For any  $1 \le l_i, l_j \leq m$, $i \neq j$, we have
	$$\mathcal{K}_{l_i} \cap \mathcal{K}_{l_j} = \emptyset.$$
\end{proposition}
\begin{proof}
	Suppose by way of contradiction that $\mathcal{K}_{l_i} \cap \mathcal{K}_{l_{j}} \not= \emptyset.$ Then there exists $v\in \mathcal{K}_{l_i} \cap \mathcal{K}_{l_j}$. We can write $v$ in two ways as follows:
	$$v=\dfrac{1}{r_{l_i}}\mathcal{V}\left(P-\dfrac{l_i}{m}I_m\right)=\dfrac{1}{r_{l_j}}\mathcal{V}\left(Q-\dfrac{l_{j}}{m}I_m\right)$$
	for some $P\in \mathcal{G}\PARENTH{l_i, \mathbb{F}^m}$ and $Q\in \mathcal{G}\PARENTH{l_j, \mathbb{F}^m}$. 
	
	Hence, $$\dfrac{1}{r_{l_i}}\left(P-\dfrac{l_i}{m}I_m\right)=\dfrac{1}{r_{l_j}}\left(Q-\dfrac{l_{j}}{m}I_m\right).$$
	Since $P$ has eigenvalues $1$ and $0$ with corresponding multiplicities $l_i$ and $m-l_i$. It follows that  $\dfrac{1}{r_{l_i}}\left(P-\dfrac{l_i}{m}I_m\right)$ has eigenvalues $\dfrac{1}{r_{l_i}}\left(1-\dfrac{l_i}{m}\right)$ and $-\dfrac{l_i}{mr_{l_i}}$ with multiplicities $l_i$ and $m-l_i$, respectively. Likewise, $\dfrac{1}{r_{l_j}}\left(Q-\dfrac{l_j}{m}I_m\right)$ has eigenvalues  $\dfrac{1}{r_{l_j}}\left(1-\dfrac{l_j}{m}\right)$ and $-\dfrac{l_j}{mr_{l_j}}$ with respective multiplicities $l_j$ and $m-l_j$. This contradicts the fact that $l_i\not=l_j$.
\end{proof}
Similar to \eqref{eq:const}, the following identity gives a relation between the  Hilbert-Schmidt inner products of orthogonal projections and the inner products of their embedded vectors.
\begin{proposition}
	Let $P$ and $Q$ be orthogonal projections of rank $l_P$ and $l_Q$, respectively. Let $v_P$ and $v_Q$ be the corresponding embedded vectors. Then 
	\begin{equation}\label{eqmixed}
	\left\langle v_P, v_Q\right\rangle=\sqrt{\dfrac{m^2}{l_Pl_Q(m-l_P)(m-l_Q)}}\left(\tr(PQ)-\dfrac{l_Pl_Q}{m}\right).
	\end{equation}
\end{proposition}  

For any mixed-rank packing $\mathcal{P}=\{P^{(l_i)}_j\}_{i=1, j=1}^{s,\ \  n_i} \in \prod_{i=1}^s \prod_{j=1}^{n_i}    \mathcal G(l_i, \mathbb F^m)$, if $v_j^{(l_i)}$ is the embedded vector corresponding to $P^{(l_i)}_j$, then  $\mathcal{C}=\{v_j^{(l_i)}\}_{i=1, j=1}^{s, \ \ n_i}$ is a code in $\mathbb{R}^{d_{\mathbb{F}^m}}$. Note that for each $i=1, 2, \ldots, s$, $\{v_j^{(l_i)}\}_{j=1}^{n_i}$ is a sequence of unit vectors lying on the compact set $\mathcal{K}_{l_i}$ of the unit sphere. In other words,
$\{v_j^{(l_i)}\}_{i=1, j=1}^{s,\ \ n_i}\in \prod_{i=1}^s \prod_{j=1}^{n_i}    \mathcal{K}_{l_i}.$

Motivated by Proposition \ref{pro1}, we will now give a definition for optimally spread of mixed-rank packings: 
\begin{definition}\label{defn1}
	Let $\mathcal{P}=\{P^{(l_i)}_j\}_{i=1, j=1}^{s,\ \  n_i}$ be a mixed-rank packing in $\mathbb{F}^m$.  $\mathcal{P}$ is said to be {\bf optimally spread} if the corresponding embedded vectors $\{v_j^{(l_i)}\}_{i=1, j=1}^{s,\ \ n_i}$ satisfy
	$$	\max_{(j, l_i) \neq (j', l_{i'})}  \langle v_j^{(l_i)}, v_{j'}^{(l_{i'})} \rangle
	= \min_{\{u_j^{(l_i)}\}_{i=1, j=1}^{s,\ \ n_i}\in \prod_{i=1}^s \prod_{j=1}^{n_i} \mathcal{K}_{l_i}}\max_{(j, l_i) \neq (j', l_{i'})}  \langle u_j^{(l_i)}, u_{j'}^{(l_{i'})} \rangle.$$
	In this case, the value $\mu:=\max_{(j, l_i) \neq (j', l_{i'})}  \langle v_j^{(l_i)}, v_{j'}^{(l_{i'})} \rangle
	$ is called the {\bf packing constant} for $\mathcal{P}$.
	
	We also say that their associated subspaces $\mathcal{W}=\{W^{(l_i)}_j\}_{i=1, j=1}^{s,\ \  n_i}$ are an optimally spread mixed-rank packing.
\end{definition}
\begin{remark}
	This definition is well-posed because each $\mathcal{K}_{l_i}$ is compact and  the objective functions are continuous.
\end{remark}

\section{Properties of optimally spread mixed-rank packings}\label{sec5}
This section is dedicated to presenting some properties of solutions of the optimal mixed-rank packing problem. In particular, as in the constant-rank case \cite{CHST}, we will see that the solutions of the mixed-rank packing problem are always fusion frames.

Suppose $\mathcal{P}=\{P^{(l_i)}_j\}_{i=1, j=1}^{s,\ \  n_i}$ is an optimally spread mixed-rank packing in $\mathbb{F}^m$. Let $\mathcal{W}=\{W^{(l_i)}_j\}_{i=1, j=1}^{s,\ \  n_i}$ be their associated subspaces. To simplify notation, in the following theorems, we will enumerate  $\mathcal{P}$ and $\mathcal{W}$ as $\mathcal{P}=\{P^{(l_i)}_i\}_{i=1}^n$ and $\mathcal{W}=\{W^{(l_i)}_i\}_{i=1}^n,$ where $n=\sum_{i=1}^{s}n_i$ is the number of subspaces of $\mathcal{W}$, $\dim W_i=l_i, 1\leq l_i\leq m.$ Note that $l_i's$ here need not to be distinct as before.

Let $\mu$ be the packing constant of an optimally spread mixed-rank packing $\mathcal{W}=\{W^{(l_i)}_i\}_{i=1}^n.$ We say that an element $W^{(l_j)}_j\in \mathcal{W}$ {\bf achieves the packing constant} if there exists some $i\in \INDEX{n}$ such that the inner product between the corresponding embedded vectors of $W^{(l_j)}_j$ and $W^{(l_i)}_i$ equals $\mu$. We call  each element $W^{(l_i)}_i$ satisfying this condition a {\bf packing neighbor} of $W^{(l_j)}_j$. 
\begin{theorem}\label{thm3}
	Let $\mathcal{W}=\{W_i^{(l_i)}\}_{i=1}^{n}, n\geq m$ be an optimally spread mixed-rank packing. Denote
	\[\mathcal{I}:=\{i:  W^{(l_i)}_i \mbox{ achieves the packing constant}\} .\]
	Then
	\[\spn\{W_i^{(l_i)}: i\in \mathcal{I}\}=\mathbb{F}^m.\]
\end{theorem}
In order to prove Theorem \ref{thm3}, we need some lemmas. The first one can be proved similarly to Lemma 3.1 in \cite{CHST}.

\begin{lemma}\label{lem1}
	Let $0<\epsilon<\alpha$ and let $\{x_i\}_{i=1}^l$ and $\{y_i\}_{i=1}^k$ be unit vectors in $\mathbb{F}^m$ satisfying:
	\[\sum_{i=1}^l \sum_{j=1}^k \left| \langle x_i, y_j\rangle \right|^2 < \alpha - \epsilon.\] 
	Let $\delta$ be such that
	\[2\sqrt{l\delta(\alpha-\epsilon)} + l\delta < \frac \epsilon 2.\]
	If
	$\{z_i\}_{i=1}^{k}$ is a sequences of unit vectors in $\mathbb F^m$ satisfying
	\[
	\sum_{i=1}^k \|z_i - y_i\|^2 < \delta,
	\] then
	\[
	\sum_{i=1}^l \sum_{j=1}^k \left| \langle x_i, z_j\rangle \right|^2 < \alpha - \frac \epsilon 2.
	\]
\end{lemma}
\begin{lemma}\label{lem2}
	Let $\mathcal{W}=\{W_i^{(l_i)}\}_{i=1}^{n}, n\geq m$ be an optimally spread  mixed-rank packing which is not an orthogonal set of lines when $n=m$. If $W^{(l_k)}_k$ is an element of $\mathcal{W}$ that achieves the packing constant, then it contains a unit vector which is not orthogonal to any of its packing neighbors.
\end{lemma}
\begin{proof} Denote
	$$\mathcal{I}_k=\{i: W^{(l_i)}_i \mbox{ is a packing neighbor of } W^{(l_k)}_k\}.$$
	Let $\mu$ be the packing constant. By the definition of packing neighbor, we have that
	\[\tr(P^{(l_i)}_iP^{(l_k)}_k)=\mu\sqrt{\dfrac{l_il_k(m-l_i)(m-l_k)}{m^2}}+\dfrac{l_il_k}{m}, \mbox{ for all } i\in \mathcal{I}_k.\]
	First, we will show that $\tr(P_i^{(l_i)}P_k^{(l_k)})>0$ for all $i\in \mathcal{I}_k$.
	
	If $n>d_{\mathbb{F}^m}+1$, then this is obvious since $\mu\geq 0$.
	
	Consider the case $m\leq n\leq d_{\mathbb{F}^m}+1.$ In this case, the Rankin bound gives $\mu\geq -\frac{1}{n-1}$.
	
	Suppose $\tr(P_i^{(l_i)}P_k^{(l_k)})=0$ for some $i$. Then 
	\[\sqrt{\dfrac{m^2}{l_il_k(m-l_i)(m-l_k)}}\left(-\dfrac{l_il_k}{m}\right)\geq -\dfrac{1}{n-1}\]
	or equivalently,
	\[(n-1)^2l_il_k\leq (m-l_i)(m-l_k).\]
	But we have
	\begin{align*}
	(n-1)^2l_il_k- (m-l_i)(m-l_k)&\geq (m-1)^2l_il_k- (m-l_i)(m-l_k)\\
	&=(m^2-2m+1)l_il_k-m^2+ml_i+ml_k-l_il_k\\
	&=(m^2-2m)l_il_k-m^2+ml_i+ml_k.
	\end{align*}
	If $l_i\not=l_k$, then we can assume that $l_i>l_k$, so $l_i\geq l_k+1$. This implies
	\begin{align*}
	(n-1)^2l_il_k- (m-l_i)(m-l_k)&\geq(m^2-2m)l_il_k-m^2+ml_i+ml_k\\
	&\geq (m^2-2m)(l_k+1)l_k-m^2+m(l_k+1)+ml_k\\
	&=ml_k^2(m-2)+m^2(l_k-1)+m>0,
	\end{align*}
	a contradiction.
	
	If $l_i=l_k\geq 2$, then 
	\begin{align*}
	(n-1)^2l_il_k- (m-l_i)(m-l_k)&\geq(m^2-2m)l_il_k-m^2+ml_i+ml_k\\
	&= (m^2-2m)l_k^2-m^2+2ml_k\\
	&=(m-1)^2l_k^2-(m-l_k)^2>0,
	\end{align*}
	which is also a contradiction.
	
	Finally, if $l_i=l_k=1$ then $n=m$, and the packing constant $\mu=-\frac{1}{m-1}$. By Theorem \ref{thm2}, the corresponding embedded vectors of the packing form a simplex. This implies every element of $\mathcal{W}$ is a packing neighbor of $W_k^{(l_k)}$. By what we have shown above, all subspaces must be 1-dimensional and pairwise orthogonal. This contradicts our assumption. 
	
	Thus, we have shown that  $\tr(P_i^{(l_i)}P_k^{(l_k)})>0$ for all $i\in \mathcal{I}_k$.
	
	Now for each $i \in \mathcal{I}_k$, let $ V_i :=  W^{(l_k)}_k \cap  [W^{(l_i)}_i]^\perp$.  Since $\tr(P^{(l_i)}_iP^{(l_k)}_k)>0$, every $V_i$ is a proper subspace of $W^{(l_k)}_k$, and since a linear space cannot be written as a finite union of proper subspaces, it follows that $ W^{(l_k)}_k \backslash \cup_{i \in \mathcal{I}_k} V_i$ is nonempty, so the claim follows.
\end{proof}
\begin{proof}[Proof of Theorem \ref{thm3}]
	If $\mathcal{W}$ contains all 1-dimensional pairwise orthogonal subspaces, then the conclusion is clear.
	
	Let $\mu$ be the packing constant. We now proceed by way of contradiction. Iteratively replacing elements of $\mathcal W$ that achieve the packing constant in such a way that we eventually obtain a new packing in $\mathbb F^m$ with the packing constant less than $\mu$, which cannot exist.
	With the contradictory approach in mind, fix a unit vector $z \in \mathbb F^m$ so that $z$ is orthogonal to all $W_i^{(l_i)}, i \in \mathcal{I}$.
	
	For a fixed $k\in \mathcal{I}$, denote
	\[\mathcal{I}_k=\left\{1\leq i\leq n: i\not=k, \tr(P^{(l_i)}_iP^{(l_k)}_k)=\mu\sqrt{\dfrac{l_il_k(m-l_i)(m-l_k)}{m^2}}+\dfrac{l_il_k}{m}\right\}.\]
	Then 
	\[\tr(P^{(l_i)}_iP^{(l_k)}_k)<\mu\sqrt{\dfrac{l_il_k(m-l_i)(m-l_k)}{m^2}}+\dfrac{l_il_k}{m}, \mbox{ for all } i\in \mathcal{I}^c_k\setminus \{k\}.\]
	
	This implies there exists $\epsilon >0$ such that 
	\[\tr(P^{(l_i)}_iP^{(l_k)}_k)<\mu\sqrt{\dfrac{l_il_k(m-l_i)(m-l_k)}{m^2}}+\dfrac{l_il_k}{m}-\epsilon, \mbox{ for all } i\in \mathcal{I}^c_k\setminus \{k\}.\]
	By Lemma \ref{lem2}, there exists a unit vector $x_{k,1}\in W^{(l_k)}_k$ which is not orthogonal to any $W_i^{(l_i)}, i\in \mathcal{I}_k$. Extend it to an orthonormal basis $\{x_{k,j}\}_{j=1}^{l_k}$ and let $\{x_{i, j}\}_{j=1}^{l_i}$ be an orthonormal basis for $W_i^{(l_i)}$, for $i \not=k$.
	
	For each $i\in \mathcal{I}_k^c\setminus \{k\}$, denote 
	\[\alpha(i):=\mu\sqrt{\dfrac{l_il_k(m-l_i)(m-l_k)}{m^2}}+\dfrac{l_il_k}{m}\] and let $\delta$ be such that
	\[0<\delta<\dfrac{1}{2} \mbox{ and } 2\sqrt{l_i\delta(\alpha(i)-\epsilon)} + l_i\delta < \frac \epsilon 2, \mbox{ for all } i\in \mathcal{I}_k^c\setminus \{k\}.\]		
	Define $$
	y_{k,1}:=\sqrt{1-\delta^2} \, x_{k,1} + \delta z \text{ and } y_{k,j}:=x_{k,j} \text { for } 2 \leq j \leq l_k,
	$$
	and define $V^{(l_k)}_k := \spn\{y_{k,j}\}_{j=1}^{l_k}$ with corresponding orthogonal projection, $Q^{(l_k)}_k$.  Because $\langle x_{k,1}, z \rangle =0$, it follows that $\{y_{k,j}\}_{j=1}^{l_k}$ is an orthonormal basis for $ V^{(l_k)}_k$.
	
	Noting that $1-\delta^2 < \sqrt{1-\delta^2}$ implies $2(1 - \sqrt{1-\delta^2}) < 2 \delta^2$, we estimate
	\[\sum_{j=1}^{l_k} \|x_{k,j} - y_{k,j} \|^2 = \|x_{k,1} - y_{k,1} \|^2 
	= 2(1 - \sqrt{1-\delta^2}) < 2\delta^2 < \delta.\]
	Hence, using Lemma \ref{lem1}, where $\alpha=\alpha(i)=\mu\sqrt{\dfrac{l_il_k(m-l_i)(m-l_k)}{m^2}}+\dfrac{l_il_k}{m}$ we have that for each $i \in \mathcal{I}_k^c\setminus \{k\}$,
	$$
	\tr(P^{(l_i)}_i Q^{(l_k)}_k) = \sum_{j=1}^{l_i} \sum_{j'=1}^{l_k} \left| \langle x_{i,j}, y_{k, j'}  \rangle  \right|^2 < \mu\sqrt{\dfrac{l_il_k(m-l_i)(m-l_k)}{m^2}}+\dfrac{l_il_k}{m} - \frac \epsilon 2.
	$$
	Furthermore, the nonorthogonality of $x_{k,1}$ with $W_i$ for every $i \in \mathcal{I}_k$ implies
	\[
	0<
	\sum_{j=1}^{l_i} \left|\langle  x_{i,j}, y_{k,1}  \rangle\right|^2 = (1 - \delta^2) \sum_{j=1}^{l_i} \left|\langle x_{i,j},  x_{k,1}  \rangle\right|^2 < \sum_{j=1}^{l_i} \left|\langle  x_{i,j},  x_{k,1}  \rangle\right|^2  \text{ for all } i \in \mathcal{I}_k.
	\]
	It follows that
	\begin{align*}
	\tr(P^{(l_i)}_i Q^{(l_k)}_k)& =
	\sum_{j=1}^{l_i} \sum_{j'=1}^{l_k} \left|\langle   x_{i,j}, y_{k,j'} \rangle \right|^2\\
	&<\sum_{j=1}^{l_i} \sum_{j'=1}^{l_k} \left|\langle   x_{i,j}, x_{k,j'} \rangle \right|^2\\
	&=\tr(P^{(l_i)}_iP_k^{(l_k)}) = \mu\sqrt{\dfrac{l_il_k(m-l_i)(m-l_k)}{m^2}}+\dfrac{l_il_k}{m}, 
	\text{ for all } i \in \mathcal I_k.
	\end{align*}
	Thus, replacing $W_k^{(l_k)}$ by $V_k^{(l_k)}$, we have that
	\[	\tr(P^{(l_i)}_i Q^{(l_k)}_k)<\mu\sqrt{\dfrac{l_il_k(m-l_i)(m-l_k)}{m^2}}+\dfrac{l_il_k}{m}, 
	\text{ for all } i \not=k,\]
	i.e., 
	\[\mu>\sqrt{\dfrac{m^2}{l_il_k(m-l_i)(m-l_k)}}\left(\tr(P^{(l_i)}_i Q^{(l_k)}_k)-\dfrac{l_il_k}{m}\right), 
	\text{ for all } i \not=k.\]
	Now, pick another $k\in \mathcal{I}$ and iterate this replacement procedure. After finite repetitions of this process, we  obtain a final packing, 
	with packing constant strictly less than $\mu$, the desired contradiction.
	This completes the proof.
\end{proof}
For any fusion frame $\mathcal{P}=\{P_i\}_{i=1}^n$, its spatial complement is the family $\mathcal{P}^{\perp}=\{I_m-P_i\}_{i=1}^n$, where $I_m$ is the identity operator in $\mathbb{F}^m$.
The next property demonstrates that optimality of a packing $\mathcal{P}$ is preserved when we switch to its spatial complement.
\begin{proposition}
	If $\mathcal{P}$ is an optimally spread mixed-rank packing with mixture $s$, then so is its spatial complement. 
\end{proposition}
\begin{proof}
	Suppose $\mathcal{P}$ is an optimally spread mixed-rank packing for $\mathbb{F}^m$. Let any $P\in \mathcal{P}$, then $Q:=I_m-P\in \mathcal{P}^\perp$. It is easy to see that $v_Q=-v_P$, where $v_P$ and $v_Q$ are the embedded vectors corresponding to $P$ and $Q$. The conclusion then follows by our definition of optimally spread mixed-rank packings.
\end{proof}
Given a mixed-rank packing $\mathcal{P}$ in $\mathbb{F}^m$, if $|\mathcal{P}|= 2d_{\mathbb{F}^m}$ and the corresponding embedded vectors form a orthoplex in $\mathbb{R}^{d_{\mathbb{F}^m}}$, then we will say that $\mathcal{P}$ is an {\bf maximal orthoplectic fusion frame}. The following is a nice property of such packings.
\begin{theorem}
	All maximal orthoplectic fusion frames are tight.
\end{theorem}		
\begin{proof}
	Suppose $\mathcal{P}=\{P_i^{(l_i)}\}_{i=1}^{2d_{\mathbb{F}^m}}$ is a maximal orthoplectic fusion frame. By definition, the corresponding embedded vectors of $\mathcal{P}$ form an orthoplex in $\mathbb{R}^{2d_{\mathbb{F}^m}}$. This implies that for every $i$, there exists $j$ such that 
	
	\begin{equation}\label{eqmax}
	\sqrt{\dfrac{m}{l_i(m-l_i)}}\left(P^{(l_i)}_i-\dfrac{l_i}{m}I_m\right)=-\sqrt{\dfrac{m}{l_j(m-l_j)}}\left(P^{(l_j)}_j-\dfrac{l_j}{m}I_m\right).
	\end{equation}
	
	Let $W_i$ and $W_j$ be the associated subspaces to $P^{(l_i)}_i$ and $P^{(l_j)}_j$, respectively. It is sufficient to show that $W_i\oplus W_j=\mathbb{F}^m$.
	
	We proceed by way of contradiction. If $\spn\{W_i, W_j\}\not=\mathbb{F}^m$, then let $x$ be orthogonal to both $W_i$ and $W_j$. 
	Using \eqref{eqmax}, we get
	
	\[-\dfrac{l_i}{m}\sqrt{\dfrac{m}{l_i(m-l_i)}}=\dfrac{l_j}{m}\sqrt{\dfrac{m}{l_j(m-l_j)}},\] which is impossible. 
	
	Likewise, if $W_i\cap W_j\not=\{0\}$, let a non-zero vector $x$ be in the intersection, then we get
	
	\[\sqrt{\dfrac{m}{l_i(m-l_i)}}\left(1-\dfrac{l_i}{m}\right)=-\sqrt{\dfrac{m}{l_j(m-l_j)}}\left(1-\dfrac{l_j}{m}\right),\] which is again impossible. 
	The conclusion then follows.
\end{proof}
\section{Constructing solutions}\label{sec6}
In this section, we will construct some infinite families of solutions for the optimal mixed-rank packing problem in $\mathbb F^m$ when the number of projections $n$ exceeds $d_{\mathbb{F}^m}+1$. By Identity \eqref{eqmixed}, we will construct packings $\mathcal{P}=\{P^{(l_i)}_j\}_{i=1, j=1}^{s,\ \  n_i} \in \prod_{i=1}^s \prod_{j=1}^{n_i}    \mathcal G(l_i, \mathbb F^m)$ such that $n=\sum_{i=1}^{s}n_i>d_{\mathbb{F}^m}+1$ and $\tr(P_j^{(l_i)}P_{j'}^{(l_{i'})})\leq \frac{l_il_{i'}}{m}$, for any $(l_i, j)\not=(l_{i'}, j')$. This implies that the corresponding embedded vectors $\{v_j^{(l_i)}\}_{i=1, j=1}^{s, \ \ n_i}$ are a solution for the usual (unrestricted) coding problem by Theorem \ref{thm2}. Thus, by definition \ref{defn1}, $\mathcal{P}$ is a solution for the optimal mixed-rank packing problem.  

Similar to the constructions in \cite{BH1, CHST}, we will construct mixed-rank packings containing coordinate projections as defined below:
\begin{definition}
	Given an orthonormal basis $\mathcal{B}=\{b_j\}_{j=1}^m$ for $\mathbb{F}^m$ and a subset $\mathcal{J}\subset \INDEX{m}$, the $\mathcal{J}$-coordinate projection with respect to $\mathcal{B}$ is $P^{\mathcal{B}}_\mathcal{J}=\sum_{j\in \mathcal{J}}b_jb_j^*$.
\end{definition} 
Our main tools for the constructions are mutually unbiased bases (MUBs) and 
block designs. 

\begin{definition}
	If $\mathcal{ B}=\{b_j\}_{j=1}^m$ and $\mathcal{B'}=\{b'_{j}\}_{j=1}^m$ are a pair of orthogonal bases for $\mathbb F^m$, then they are mutually unbiased bases if $|\langle b_j, b'_{j'}\rangle|^2=\frac{1}{m}$ for every $j, j'\in \INDEX{m}.$ A collection of orthonormal bases $\{\mathcal{B}_k\}_{k\in K}$ is called a set of mutually unbiased bases (MUBs) if the pair $\mathcal{B}_k$ and $\mathcal{B}_{k'}$ is mutually unbiased for every $k\not=k'$.
\end{definition}
As in \cite{BH1, CHST}, maximal sets of MUBs play an important role in our constructions. The following theorems give an upper bound for the cardinality of the set of MUBs in terms of $m$, and sufficient conditions to attain this bound.
\begin{theorem} [\cite{DGS}]
	If $\{B_k\}_{k\in K}$ is a set of MUBs for $\mathbb F^m$, then $|K|\leq m/2+1$ if $\mathbb{F}=\mathbb R$, and $|K|\leq m+1$ if $\mathbb{F}=\mathbb{C}$.
\end{theorem}
\begin{theorem} [\cite{CS, WF}]\label{maximal MUBs}
	If $m$ is a prime power, then a family of $m+1$ pairwise mutually unbiased for $\mathbb{C}^m$ exists. If $m$ is a power of $4$, then a family of $m/2+1$ pairwise mutually unbiased for $\mathbb{R}^m$ exists.
\end{theorem}
Henceforth, we abbreviate $k_{\mathbb{R}^m}=m/2+1$ and $k_{\mathbb{C}^m}=m+1$.
In order to construct the desired packings, the following simple proposition is very useful.
\begin{proposition}\label{pro3}
	\begin{enumerate}
		\item Let $\mathcal{B}=\{b_j\}_{j=1}^m$ be an orthonormal basis for $\mathbb{F}^m$. Then for any subsets $\mathcal{J}, \mathcal{J}'\in \INDEX{m}$, we have 
		\[\tr(P^{\mathcal{B}}_\mathcal{J}P^{\mathcal{B}}_{\mathcal{J}'})= |\mathcal{J} \cap \mathcal{J}'|.\]
		\item Let $\mathcal{B}_1=\{x_j\}_{j=1}^m$ and $\mathcal{B}_2=\{y_j\}_{j=1}^m$ be a pair of MUBs for $\mathbb{F}^m$.
		Then for any subsets $\mathcal{J}, \mathcal{J}'\in \INDEX{m}$, we have 
		\[\tr(P^{\mathcal{B}_1}_\mathcal{J}P^{\mathcal{B}_2}_{\mathcal{J}'})= \dfrac{|\mathcal{J}||\mathcal{J}'|}{m}.\]
	\end{enumerate}
\end{proposition}
\begin{proof}

	We compute
	\[\tr(P^{\mathcal{B}}_\mathcal{J}P^{\mathcal{B}}_{\mathcal{J}'})=\sum_{j\in \mathcal{J}}\sum_{j'\in \mathcal{J}'}\tr(b_jb_j^*b_{j'}b_{j'}^*)=\sum_{j\in \mathcal{J}}\sum_{j'\in \mathcal{J}'}|\langle b_j,b_{j'}\rangle|^2=|\mathcal{J} \cap \mathcal{J}'|,\] which is (1). 
	
	For (2), since $\mathcal{B}_1$ and $\mathcal{B}_2$ are mutually unbiased, it follows that $|\langle x_j, y_{j'}\rangle|^2=1/m$, for every $j\in \mathcal{J}, j'\in \mathcal{J}'$. Therefore,  
	\[\tr(P^{\mathcal{B}_1}_JP^{\mathcal{B}_2}_{J'})=\sum_{j\in \mathcal{J}}\sum_{j'\in \mathcal{J}'}\tr(x_jx_j^*y_{j'}y_{j'}^*)=\sum_{j\in \mathcal{J}}\sum_{j'\in \mathcal{J}'}|\langle x_j,y_{j'}\rangle|^2=\dfrac{|\mathcal{J}||\mathcal{J}'|}{m},\] which is the claim.
\end{proof}

To control the trace between coordinate projections generated by the same orthonormal basis, we need the following concept stated in \cite{BH1}.

\begin{definition}
	Let $\mathbb{S}$ be a collection of subsets of $\INDEX{m}$ such that each $\mathcal{J}$ of $\mathbb{S}$ has the same cardinality. We say that $\mathbb{S}$ is $c$-cohesive if there exists $c>0$ such that $\underset{\mathcal{J}, \mathcal{J}'\in \mathbb{S}, \mathcal{J}\not=\mathcal{J}'}{\max}| \mathcal{J}\cap \mathcal{J}'|\leq c.$
\end{definition}

Another ingredient for our constructions is block designs.
\begin{definition}
	A $t$-$(m, l, \lambda)$ block design is a collection $\mathbb{S}$ of subsets of $\INDEX{m}$, called blocks, where each block $\mathcal{J}\in \mathbb{S}$ has cardinality $l$, such that every subset of $\INDEX{m}$ with cardinality $t$ is contained in exactly $\lambda$ blocks and each element of $\INDEX{m}$ occurs in exactly $r$ blocks. We also denote $b=|\mathbb{S}|$, the cardinality of $\mathbb{S}$. If $t=2$, then the design is called a {\bf balanced incomplete block design or BIBD}. When the parameters are not important or implied by the context, then $\mathbb{S}$ is also referred to as a $t$-block design.
\end{definition}
The following proposition gives a few simple facts about block designs.
\begin{proposition}
	For any $t$-$(m, l, \lambda)$ block design, the following conditions hold:
	\begin{enumerate}
		\item $mr=bl$,
		\item $r(l-1)=\lambda(m-1)$ if $t=2$.
	\end{enumerate}
	Furthermore, a $t$-block design is also a $(t-1)$-block design for $t>1$.
\end{proposition}
Tight fusion frames can be formed by coordinate projections from block designs.
\begin{proposition}(\cite{BH1, CHST})\label{tighness}
	Let $\mathcal{B}=\{b_j\}_{j=1}^m$ be an orthonormal basis for $\mathbb{F}^m$ and $\mathbb{S}=\{\mathcal{J}_1, \mathcal{J}_2, \ldots, \mathcal{J}_b\}$ be a $t$-$(m, l, \lambda)$-block design. Then the set of coordinate projections with respect to $\mathcal{B}$, $\{P^{\mathcal{B}}_{\mathcal{J}_1}, P^{\mathcal{B}}_{\mathcal{J}_2}, \ldots, P^{\mathcal{B}}_{\mathcal{J}_b}\}$, forms a tight fusion frame for $\mathbb{F}^m$.
\end{proposition}
\begin{proof}
	Since $\mathbb{S}$ is a $t$-block design, every element $j\in \INDEX{m}$ occurs in exactly $r$ blocks. It follows that
	\[\sum_{j=1}^{b}P^{\mathcal{B}}_{\mathcal{J}_j}=r\sum_{j=1}^{m}b_jb_j^*=rI_m.\]
	This means that $\{P^{\mathcal{B}}_{\mathcal{J}_j}\}_{j=1}^b$ forms a tight fusion frame.
\end{proof}
The main idea of our constructions is using a set of MUBs and a collection of block designs to form a packing of coordinate projections. We are now ready for the main theorem of the constructions.
\begin{theorem}\label{thm_construction}
	Let  $\{\mathcal{B}_k\}_{k\in K}$ be a family of MUBs in $\mathbb{F}^m$, and let $(m, l_1, \lambda_1), \ldots, (m, l_s, \lambda_s)$, $s\leq |K|$ be a family of block designs. Denote $\mathbb{S}_1, \mathbb{S}_2, \ldots, \mathbb{S}_s$ the corresponding sets of blocks. Let $\{A_1, A_2, \ldots, A_s\}$ be any partition of $K$ and for each $i\in \INDEX{s}$, let $\mathcal{P}_i=\{P^{\mathcal{B}_k}_{\mathcal{J}}: k\in A_i, \mathcal{J}\in \mathbb{S}_i\}$ be the family of coordinate projections of rank $l_i$. If for each $i$,  $\mathbb{S}_i$ is $l_i^2/m$-cohesive and $\sum_{i=1}^{s}|\mathbb{S}_i||A_i|>d_{\mathbb{F}^m}+1$, then the union of families $\{\mathcal{P}_i\}_{i=1}^s$ forms a tight optimally spread mixed-rank packing with mixture $s$ for $\mathbb{F}^m$.
\end{theorem}
\begin{proof}
	For every fixed $i\in \INDEX{s}$, consider any coordinate projections $P_\mathcal{J}^{\mathcal{B}_k},P_{\mathcal{J}'}^{\mathcal{B}_{k'}}$ in $\mathcal{P}_i$. If $k=k'$ then by (1) of Proposition \ref{pro3} and the assumption that $\mathbb{S}_i$ is $l_i^2/m$-cohesive, we get
	$$\tr(P^{(k)}_\mathcal{J}P^{(k')}_{\mathcal{J}'})=|\mathcal{J}\cap \mathcal{J}'|\leq l_i^2/m.$$  If $k\not=k'$, then by (2) of Proposition \ref{pro3}, we have 
	$$\tr(P^{(k)}_\mathcal{J}P^{(k')}_{\mathcal{J}'})=l_i^2/m.$$
	Thus, for each $i$, the family of coordinate projections of rank $l_i$, $\mathcal{P}_i=\{P^{\mathcal{B}_k}_{\mathcal{J}}: k\in A_i, \mathcal{J}\in \mathbb{S}_i\}$ forms a packing of $|\mathbb{S}_i||A_i|$ elements and the trace of the product of any two projections is at most $l_i^2/m$. 
	
	Now consider any $P\in \mathcal{P}_i$ and $Q\in \mathcal{P}_j$ where $i\not=j$. Then again,  by (2) of Proposition \ref{pro3},  $\tr(PQ)=\dfrac{l_il_j}{m}$. Moreover, by assumption, the number of elements of the union is greater than $d_{\mathbb{F}^m}+1$. Therefore, this family forms an optimally spread mixed-rank packing with a mixture of $s$. The tightness is clear by Proposition \ref{tighness}.
\end{proof}

Symmetric block designs are a special type of block design and are well-studied objects in design theory, for example, see \cite{CL, CD}. In \cite{CHST}, the authors exploited their nice properties to construct optimally spread packings of constant-rank. In this paper, we will continue to use them for our constructions.
\begin{definition}
	A 2-$(m, l, \lambda)$ block design is symmetric if $m=b$, or equivalently $l=r$.
\end{definition}
A very useful property of symmetric block designs is that the pairwise block intersections have the same cardinality.
\begin{theorem} [\cite{CL}]
	For a $(m, l, \lambda)$ symmetric block design $\mathbb{S}$, every $\mathcal{J}, \mathcal{J}'\in \mathbb{S}, \mathcal{J}\not=\mathcal{J}'$ satisfies 
	\[|\mathcal{J}\cap \mathcal{J}'|=\lambda.\]
\end{theorem} 
Moreover, it has been shown in \cite{CHST} that every symmetric block design has cohesive property.
\begin{proposition} \label{cohesive}
	Every  $(m, l, \lambda)$ symmetric block design is $l^2/m$-cohesive.
\end{proposition}

We also need another property of symmetric block designs. 
\begin{proposition} [\cite{CD}]
	The complement of a  $(m, l, \lambda)$ symmetric design $\mathbb{S}$ is a $(m, m-l, m-2l+\lambda)$ symmetric design $\mathbb{S}^c$, whose blocks are the complements of blocks of $\mathbb{S}$.
\end{proposition}

The following theorem will consider a special case where all block designs are symmetric.

\begin{theorem}\label{thmsymmetric}
	Let  $\{\mathcal{B}_k\}_{k\in K}$ be a family of MUBs and let $(m, l_1, \lambda_1), \ldots, (m, l_s, \lambda_s)$, $s\leq |K|$ be a family of symmetric block designs. Denote $\mathbb{S}_1, \mathbb{S}_2, \ldots, \mathbb{S}_s$ the corresponding sets of blocks. Let $\{A_1, A_2, \ldots, A_s\}$ be any partition of $K$ and let $\mathcal{P}_i=\{P^{\mathcal{B}_k}_{\mathcal{J}}: k\in A_i, \mathcal{J}\in \mathbb{S}_i\}$ be the family of coordinate projections of rank $l_i$. If $m\sum_{i=1}^{s}|A_i|>d_{\mathbb{F}^m}+1$, then the union of families $\{\mathcal{P}_i\}_{i=1}^s$ forms a tight optimally spread mixed-rank packing with a mixture of $s$ for $\mathbb{F}^m$.
\end{theorem}
\begin{proof}
	For each $i\in \INDEX{s}$, by Proposition \ref{cohesive},  $\mathbb{S}_i$ is $l_i^2/m$-cohesive. Note that all the sets $\mathbb{S}_i$ have the same cardinality, $m$. The conclusion then follows by Theorem \ref{thm_construction}. 
\end{proof}
\begin{example}{Example}
	\begin{enumerate}
		\item  Let $\{\mathcal{B}_1, \ldots, \mathcal{B}_9\}$ be 9 MUBs in $\mathbb{R}^{16}$ and let $\mathbb{S}$ be a $(16, 6, 2)$ symmetric  block design, which exists, see \cite{CD}. Let $\{P^{\mathcal{B}_1}_i\}_{i=1}^{16}$ be 16 projections on lines, each spanned by a vector in $\mathcal{B}_1$ and let $\{P^{\mathcal{B}_k}_\mathcal{J} :  \mathcal{J}\in \mathbb{S}\}$ be the collection of coordinate projections respective to $\mathcal{B}_k$, for $k=2, \ldots, 9.$ Then the family
		\[\{P^{\mathcal{B}_1}_i\}_{i=1}^{16} \cup \{P^{\mathcal{B}_k}_\mathcal{J} :  \mathcal{J}\in \mathbb{S}\}_{k=2}^9\] forms a tight optimally spread mixed-rank packing in $\mathbb{R}^{16}$ with a mixture of 2. This family has 16 rank 1 projections and 128 rank 6 projections.
		\item	Let $\mathbb{S}_1$ and  $\mathbb{S}_2$ be symmetric block designs with parameters $(71, 15, 3)$ and $(71, 21, 6)$, respectively,  see \cite{CD}. Since 71 is a prime, there are 72 MUBs in $\mathbb{C}^{71}.$ Let $\{A_1, A_2\}$ be any partition of \INDEX{72}. Then the family
		$$\cup_{i=1}^2\{P^{\mathcal{B}_k}_{\mathcal{J}}: k\in A_i, \mathcal{J}\in \mathbb{S}_i\}$$ forms a tight optimally spread mixed-rank packing in $\mathbb{C}^{71}$.  It has $71\times72$ elements, namely $71|A_1|$ elements of rank 15 and $71|A_2|$ elements of rank 21. 
	\end{enumerate}
\end{example}	
As a consequence of Theorem \ref{thmsymmetric} and noting that maximal sets of MUBs exist by Theorem \ref{maximal MUBs}, we get:
\begin{corollary}\label{sym2}
	Let $\mathbb{S}$ be a symmetric block design and $\mathbb{S}^c$ be its complement. 
	\begin{enumerate}
		\item 	If $m$ is a power of 4, then for any partition $\{A_1, A_2\}$ of $\INDEX{m/2+1}$, the family 
		$$\{P^{\mathcal{B}_k}_{\mathcal{J}}: k\in A_1, \mathcal{J}\in \mathbb{S}\}\cup \{P^{\mathcal{B}_k}_{\mathcal{J}}: k\in A_2, \mathcal{J}\in\mathbb{S}^c\}$$ forms a tight optimally spread mixed-rank packing with a mixture of 2 in $\mathbb{R}^m$, where $\{\mathcal{B}_k\}_{k=1}^{m/2+1}$ is a maximal set of MUBs in $\mathbb{R}^m$.
		\item 	If $m$ is a prime power, then for any partition $\{A_1, A_2\}$ of $\INDEX{m+1}$, the family 
		$$\{P^{\mathcal{B}_k}_{\mathcal{J}}: k\in A_1, \mathcal{J}\in \mathbb{S}\}\cup \{P^{\mathcal{B}_k}_{\mathcal{J}}: k\in A_2, \mathcal{J}\in\mathbb{S}^c\}$$ forms a tight optimally spread mixed-rank packing with a mixture of 2 in $\mathbb{C}^m$, where $\{\mathcal{B}_k\}_{k=1}^{m+1}$ is a maximal set of MUBs in $\mathbb{C}^m$.
	\end{enumerate}
	
\end{corollary}

Besides symmetric block designs, affine designs are also very useful for our constructions.
\begin{definition}
	A  $2$-$(m, l, \lambda)$ block design is resolvable if its collection of blocks $\mathbb{S}$ can be partitioned into subsets, called parallel classes, such that:
	\begin{enumerate}
		\item the blocks within each class are disjoint, and
		\item for each parallel class, every element of $\INDEX{m}$ is contained in a block.
	\end{enumerate}
	Moreover, if the number of elements occurring in the intersection between blocks from different parallel classes is constant, then it is called an affine design.
\end{definition}
For any $2$-$(m, l, \lambda)$ resolvable block design, Bose's condition gives a lower bound for the number of blocks $b$.
\begin{theorem} (\cite{CD})
	Given any $2$-$(m, l, \lambda)$ resolvable block design, the number of blocks is bounded by the other parameters:
	\[b\geq m+r-1,\]
	and this lower bound is achieved if and only if the design is a $l^2/m$-cohesive affine design.
\end{theorem}
\begin{proposition}\label{affine}
	If $\mathbb{S}$ is a $l^2/m$-cohesive $(m, l, \lambda)$ affine design, then its complement $\mathbb{S}^c$ is $(m-l)^2/m$-cohesive.
\end{proposition}
\begin{proof}
	Let any $\mathcal{J}_1, \mathcal{J}_2\in \mathbb{S}$. If they are in the same parallel class, then $\mathcal{J}_1\cap \mathcal{J}_2=\emptyset$. Hence
	\[|\mathcal{J}^c_1\cap \mathcal{J}^c_2 |=|\INDEX{m}\setminus (\mathcal{J}_1\cup \mathcal{J}_2)|=m-2l<\dfrac{(m-l)^2}{m}.\]
	If $\mathcal{J}_1$ and $\mathcal{J}_2$ are in different parallel classes, then by definition, 
	$|\mathcal{J}_1\cap \mathcal{J}_2|=\dfrac{l^2}{m}$. Therefore, 
	$|\mathcal{J}_1\cup \mathcal{J}_2|=2l-l^2/m$ and so
	\[|\mathcal{J}^c_1\cap \mathcal{J}^c_2 |=m-2l+\dfrac{l^2}{m}=\dfrac{(m-l)^2}{m}.\]
	This completes the proof.
\end{proof}
In the following, we will give some infinite families of optimally spread mixed-rank packings. Before giving such examples, we recall some block designs from \cite{CD}.
\begin{itemize}
	\item Affine designs with parameters
	$$(m, l, \lambda)=\left(p^{t+1}, p^t, \frac{p^{t}-1}{p-1}\right), \ t\geq 1$$ exist if $p$ is a prime power.
	
	\item Menon symmetric designs with parameters $(m, l,\lambda)=(4t^2, 2t^2-t, t^2-t), t\geq 1$. These designs exist for instance, when Hadamard matrices of order $2t$ exist, and it is conjectured that they exist for all
	values of $t$.

	\item Hadamard symmetric designs with parameters $(m, l, \lambda)=(4t-1, 2t-1, t-1), t\geq 1$. These designs exist if and only if Hadamard matrices of order $4t$ exist.
	
\end{itemize} 			
\begin{example}{Example}
	Let $\mathbb{S}$ be an affine design with parameters $\left(p^{t+1}, p^t, \frac{p^{t}-1}{p-1}\right)$, where $p$ is prime power. Note that we can verify the values of the remaining parameters 
	\[r=\sum_{i=0}^{t}p^i, \mbox{ and } b=\sum_{i=1}^{t+1}p^i.\]
	According to Bose's condition, this design is $p^{t-1}$-cohesive. Let $\mathbb{S}^c$ be its complement. By Proposition \ref{affine}, it is $p^{t-1}(p-1)^2$-cohesive. Since $p^{t+1}$ is a prime power, a maximal set of MUBs, $\{\mathcal{B}_k\}_{k=1}^{p^{t+1}+1}$, for $\mathbb{C}^{p^{t+1}} $ exists. By Theorem \ref{thm_construction}, for any partition $\{A_1, A_2\}$ of $\INDEX{p^{t+1}+1}$, the family 
	$$\{P^{\mathcal{B}_k}_{\mathcal{J}}: k\in A_1, \mathcal{J}\in \mathbb{S}\}\cup \{P^{\mathcal{B}_k}_{\mathcal{J}}: k\in A_2, \mathcal{J}\in\mathbb{S}^c\}$$ forms a tight optimally spread mixed-rank packing with a mixture of 2 in $\mathbb{C}^{p^{t+1}}.$		
\end{example}	
\begin{example}{Example}
	Let $\mathbb{S}_1$ be an affine design with parameters
	$\left(4^{t+1}, 4^t, \frac{4^{t}-1}{4-1}\right), \ t\geq 1$ and $\mathbb{S}_2$ be its complement.
	Let $\mathbb{S}_3$ be the Menon symmetric design with parameters
	$(4^{t+1}, 2^{2t+1}-2^t, 2^{2t}-2^t)$. Denote $\mathbb{S}_4$ the complement of $\mathbb{S}_3$.
	
	Let $\{\mathcal{B}_k\}_{k=1}^{4^{t+1}/2+1}$ be a maximal set of MUBs which exists in $\mathbb{R}^{4^{t+1}}$. By Theorem \ref{thm_construction}, for any partition $\{A_1, A_2, A_3, A_4\}$ of $\INDEX{4^{t+1}/2+1}$, the family 
	$$\mathcal{P}=\cup_{i=1}^4\{P^{\mathcal{B}_k}_{\mathcal{J}}: k\in A_i, \mathcal{J}\in \mathbb{S}_i\}$$ forms a tight optimally spread mixed-rank packing with a mixture of 4 in $\mathbb{R}^{4^{t+1}}$.	
\end{example}

Finally, we will construct optimally spread mixed-rank packings in  $\mathbb{F}^m$ whose corresponding embedded vectors are vertices of an orthoplex in $\mathbb{R}^{d_{\mathbb{F}^m}}$.
\begin{theorem}\label{maximal}
	Let $\mathbb{S}$ be a collection of subsets of $\INDEX{m}$, each of size $l$. Suppose that $|\mathcal{J}\cap \mathcal{J}'|=l^2/m$ for any $\mathcal{J}, \mathcal{J}'\in \mathbb{S}$. Let $\mathbb{S}^c$ be the complement of $\mathbb{S}$, i.e., $$\mathbb{S}^c:=\{\mathcal{J}^c=\INDEX{m}\setminus \mathcal{J}:  \mathcal{J}\in \mathbb{S}\}.$$ Let $\{\mathcal{B}_{k}\}_{k\in K}$ be a set of MUBs for $\mathbb{F}^m$ and let $\mathcal{P}=\{P^{\mathcal{B}_k}_{\mathcal{J}}: k\in K, \mathcal{J}\in \mathbb{S}\cup \mathbb{S}^c\}$. If $2|\mathbb{S}||K|>d_{\mathbb{F}^m}+1$, then $\mathcal{P}$ is an optimally spread mixed-rank packing for which the embedded vectors corresponding to the coordinate projections in $\mathcal{P}$ occupy the vertices of an orthoplex in $\mathbb{R}^{d_{\mathbb{F}^m}}$.
	Furthermore, if $|\mathbb{S}|=m-1$ and a maximal set of MUBs, $\{\mathcal{B}_k\}_{k=1}^{k_{\mathbb{F}^m}}$, exists in $\mathbb{F}^m$, then the collection of projections  $\mathcal{P}=\{P_{\mathcal{J}}^{\mathcal{B}_k}: k\in \INDEX{k_{\mathbb{F}^m}}, \mathcal{J}\in \mathbb{S}\cup \mathbb{S}^c\}$ is a maximal orthoplectic fusion frame for $\mathbb{F}^m$.
	
\end{theorem}
\begin{proof}
	Since  $|\mathcal{J}\cap \mathcal{J}'|=l^2/m$, for every $\mathcal{J}, \mathcal{J}'\in \mathbb{S}$, it is easy to see that $|\mathcal{J}\cap \mathcal{J}'|=\frac{(m-l)^2}{m}$, for every $\mathcal{J}, \mathcal{J}'\in \mathbb{S}^c$. 
	
	Moreover, for every $\mathcal{J}\in \mathbb{S}, \mathcal{J}'\in \mathbb{S}^c, \mathcal{J}'\not=\mathcal{J}^c$, we have
	\[|\mathcal{J}\cap \mathcal{J}'|=|\mathcal{J}\setminus (\mathcal{J}')^c|=l-\frac{l^2}{m}.\] 
	Hence, for each $k\in K$, by (1) of Proposition \ref{pro3}, it follows that
	\[\tr(P^{\mathcal{B}_k}_{\mathcal{J}}P^{\mathcal{B}_k}_{\mathcal{J}'})-\dfrac{\rank(P^{\mathcal{B}_k}_{\mathcal{J}})\rank(P^{\mathcal{B}_k}_{\mathcal{J}'})}{m}=0, \mbox{ for every } \mathcal{J}, \mathcal{J}'\in \mathbb{S}\cup\mathbb{S}^c, \mathcal{J}'\not=\mathcal{J}^c.\]
	Note also that for every $\mathcal{J}\in \mathbb{S}$, the corresponding embedded vectors of $P^{\mathcal{B}_k}_{\mathcal{J}}$ and $P^{\mathcal{B}_k}_{\mathcal{J}^c}$ are antipodal points in $\mathbb{R}^{d_{\mathbb{F}^m}}$. Moreover, by (2) of Proposition \ref{pro3}, for any $k, k'\in K, k\not=k'$, we have 
	
	\[\tr(P^{\mathcal{B}_{k}}_{\mathcal{J}}P^{\mathcal{B}_{k'}}_{\mathcal{J}'})-\dfrac{\rank(P^{\mathcal{B}_k}_{\mathcal{J}})\rank(P^{\mathcal{B}_{k'}}_{\mathcal{J}'})}{m}=0, \mbox{ for every } \mathcal{J}, \mathcal{J}'\in \mathbb{S}\cup\mathbb{S}^c.\]
	
	Note that $|\mathbb{S}\cup \mathbb{S}^c|=2|\mathbb{S}|$ and so there are $2|\mathbb{S}|$ coordinate projections respective to each $\mathcal{B}_k, k\in K$. Hence, the total number of projections is $|\mathcal{P}|=2|\mathbb{S}||K|>d_{\mathbb{F}^m}+1$ by assumption.
	The conclusion of the first part of the theorem then follows by using identity \eqref{eqmixed}.
	
	The ``furthermore part'' is obvious since in this case, 
	\[|\mathcal{P}|=2(m-1)k_{\mathbb{F}^m}=2d_{\mathbb{F}^m}.\]
	This completes the proof.
\end{proof}

Recall that for any $(m, l, \lambda)$ symmetric block design, the intersection between any of its blocks has exactly $\lambda$ elements. Note that $\lambda=\frac{l(l-1)}{m-1}<\frac{l^2}{m}$. However, if $\frac{m-l}{l-1}$ is an integer, then we can view these blocks as subsets of a bigger set, $\INDEX{m'}$ so that $\lambda=\frac{l^2}{m'}$. In other words, these blocks, viewed as subsets of $\INDEX{m'}$, satisfy the assumption of Theorem \ref{maximal}. We will record this by the following proposition.

\begin{proposition}
	Let $\mathbb{S}$ be a $(m, l, \lambda)$ symmetric block design. Suppose that $\frac{m-l}{l-1}$ is an integer. Let $m'=m+\frac{m-l}{l-1}$. Then for any $\mathcal{J}, \mathcal{J}'\in \mathbb{S}$, we have $|\mathcal{J}\cap \mathcal{J}'|=\frac{l^2}{m'}$.
\end{proposition}

\begin{example}{Example}
	Let $q$ be a prime power and let $\mathbb{S}$ be a symmetric block design with parameters
	$$m=q^2+q+1, \ l=q+1,\  \lambda=1.$$ 
	Such designs exist, called Projective planes, see \cite{CD}. We have $\frac{m-l}{l-1}=q$, which is an integer. Let $m'=m+q=(q+1)^2$. If $m'$ is a prime power, then by Theorem \ref{maximal MUBs}, a maximal set of MUBs, $\{\mathcal{B}_k\}_{k=1}^{m'+1}$ exists in $\mathbb{C}^{m'}$. By Theorem \ref{maximal}, the family 
	\[\mathcal{P}=\{P_{\mathcal{J}}^{\mathcal{B}_k}: k\in \INDEX{m'+1}, \mathcal{J}\in \mathbb{S}\cup \mathbb{S}^c\},\] where $\mathbb{S}^c=\{\INDEX{m'}\setminus \mathcal{J}: \mathcal{J}\in \mathbb{S}\}$, is an optimally spread mixed-rank packing for $\mathbb{C}^{m'}$ with a mixture of 2. A similar result is obtained for $\mathbb{R}^{m'}$ when $m'$ is a power of 4. However, in both cases, the projections cannot embed exhaustively into every vertex of the orthoplex since $|\mathcal{P}|=2mk_{\mathbb{F}^{m'}}<2d_{\mathbb{F}^{m'}}.$
	
\end{example}

It is known that an extension of a $(4t-1, 2t-1, t-1)$ Hadamard design is a $(4t, 2t, t-1)$ Hadamard 3-design. This design exists if and only if a Hadamard matrix of order $4t$ exists. Moreover, it can be constructed from a Hadamard matrix as follows. 

Let $H=(h_{ij})$ be a Hadamard matrix of order $4t$. Normalize $H$ so that all elements in the last row are $+1$. For each row $i$ other than the last, we define a pair of subsets of $\INDEX{4t}$ by
\[\mathcal{J}_i=\{j\in \INDEX{4t}: h_{ij}=+1\}, \mbox{ and } \mathcal{J}'_i=\{j\in \INDEX{4t}: h_{ij}=-1\}.\]
Then the collection $\mathbb{S}$ of all these subsets forms a Hadamard 3-design. Note that $\mathcal{J}'_i=\mathcal{J}^c_i$, for all $i$. Hence if we let $\mathbb{S}_1=\{\mathcal{J}_i : i=1, 2, \ldots, 4t-1\}$, then the set of blocks of a Hardmard 3-design has the form $\mathbb{S}=\mathbb{S}_1\cup
\mathbb{S}_1^c$. One of the useful properties of a $(4t, 2t, t-1)$ Hadamard 3-design is that the cardinality of the intersection of any of its two blocks $\mathcal{J}, \mathcal{J}'$, where $\mathcal{J}'\not=\mathcal{J}^c$ is the same, namely, $|\mathcal{J}\cap \mathcal{J}'|=t$. For more properties of these designs, see for example, \cite{CL, CD}.

It turns out that we can use Hadamard 3-designs and maximal sets of mutually unbiased bases to construct maximal orthoplectic fusion frames of constant-rank. Moreover, any such fusion frames constructed in this way must come from Hadamard 3-designs. This does not seem to be mentioned in previous papers \cite{BH1, CHST}. Note that in \cite{BH1}, the authors give a construction of a family of block designs and then use them to construct a family of maximal orthoplectic fusion frames of constant-rank. They also claim that these block designs are 2-designs. However, as a consequence of the following theorem, they are actually Hadamard 3-designs.
\begin{theorem}
	Let $\mathbb{S}$ be a collection of subsets of $\INDEX{m}$, each of size $l$. Let $\{\mathcal{B}_{k}\}_{k\in K}$ be a maximal set of MUBs for $\mathbb{F}^m$, with $|K|=k_{\mathbb{F}^m}$, and let $\mathcal{P}=\{P^{\mathcal{B}_k}_{\mathcal{J}}: k\in K, \mathcal{J}\in \mathbb{S}\}$. If $\mathbb{S}$ is a Hadamard 3-design, then $\mathcal{P}$ is a maximal orthoplectic fusion frame of constant-rank for $\mathbb{F}^m.$ In particular,
	\begin{enumerate}
		\item If $m$ is a power of 2, then $\mathcal{P}$ is a maximal orthoplectic fusion frame for $\mathbb{C}^m$.
		\item If $m$ is a power of $4$, then $\mathcal{P}$ is a maximal orthoplectic fusion frame for $\mathbb{R}^m$.
	\end{enumerate}
	Conversely, if $\mathcal{P}$ is a maximal orthoplectic fusion frame of constant-rank for $\mathbb{F}^m$, then $\mathbb{S}$ is a Hadamard 3-design.
\end{theorem}
\begin{proof}
	Suppose $\mathbb{S}$ is a Hardmard 3-design of parameters $(m, l,\lambda)=(4t, 2t, t-1)$. Then $|\mathbb{S}|=8t-2=2(m-1)$. Note that $l=2t=m/2$, so every projection of $\mathcal{P}$ has the same rank. Moreover, $|\mathcal{J}\cap \mathcal{J}'|=t=l^2/m$, for every distinct blocks $\mathcal{J}, \mathcal{J}'\in \mathbb{S}, \mathcal{J}'\not=\mathcal{J}^c$. Hence, $$	\mu\PARENTH{\mathcal P}  = \max\{\tr(P^{\mathcal{B}_k}_{\mathcal{J}}P^{\mathcal{B}_{k'}}_{\mathcal{J}'}): \mathcal{J}, \mathcal{J}'\in \mathbb{S}, \mathcal{J} \neq \mathcal{J}', k, k'\in K \}=l^2/m.$$ The conclusions then follow from Theorem \ref{thm1}.
	
	Conversely, suppose $\mathcal{P}$ is a maximal orthoplectic fusion frame of constant-rank for $\mathbb{F}^m$. Then each projection of $\mathcal{P}$ must have rank $l=m/2$, (see Corollary 2.7 in \cite{BH1}), and hence every $\mathcal{J}\in \mathbb{S}$ has size $m/2$. Moreover, $\mathbb{S}$  is of size $2d_{\mathbb{F}^m}/k_{\mathbb{F}^m}=2(m-1)$. Since $\mathcal{P}$ is maximal, it follows from \eqref{eq:const} that for each $\mathcal{J}\in \mathbb{S}$, its complement $\mathcal{J}^c$ is also in $\mathbb{S}$. Thus, $\mathbb{S}$ is of the form $\mathbb{S}=\mathbb{S}_1\cup \mathbb{S}_1^c$. Furthermore, by \eqref{eq:const}, $\tr(P^{\mathcal{B}_k}_{\mathcal{J}}P^{\mathcal{B}_{k'}}_{\mathcal{J}'})=l^2/m=m/4$, for any $k, k'\in K$, and $\mathcal{J}, \mathcal{J}'\in \mathbb{S}, \mathcal{J}'\not=\mathcal{J}^c.$ 
	
	Denote the subsets of $\mathbb{S}_1$ by $\mathcal{J}_1, \mathcal{J}_2, \ldots, \mathcal{J}_{m-1}$. Let $H$ be a $m\times m$-matrix whose the last row is of all +1, and the entries of row $i$ is defined by $H_{ij}=1$ if $j\in \mathcal{J}_i$, and $H_{ij}=-1$, otherwise. We will show that $H$ is a Hadamard matrix and therefore $\mathbb{S}$ is a Hadamard 3-design. By the construction, each row other than the last has precisely $m/2$ entries +1 and $m/2$ entries -1. Therefore, it is enough to show that the inner product of any two of them is zero. 
	
	Indeed, take any two distinct rows $R_i$ and $R_{i'}$ of $H$, $i, i'\leq m-1$. Since $$|\mathcal{J}_i\cap \mathcal{J}_{i'}|=\tr(P^{\mathcal{B}_k}_{\mathcal{J}_i}P^{\mathcal{B}_{k}}_{\mathcal{J}_{i'}})=m/4, \mbox{ for all } k \in K,$$ it follows that row $R_i$ and row $R_{i'}$ have exactly $m/4$ entries of +1 in the same column. This implies $\langle R_i, R_{i'}\rangle=0$, which completes the proof. 
\end{proof}
{\bf Acknowledgement:}
	The authors were supported by NSF DMS 1609760, 1906725, and NSF ATD 1321779.

\end{document}